\documentclass{article}

\usepackage{amsmath,amssymb,amsthm}
\usepackage{tikz}
\usepackage{color}
\usepackage[toc]{appendix}
\usepackage{graphicx}
\usepackage{fancyhdr}
\usepackage{enumitem}
\usepackage{bbm}
\usepackage{parskip}
\usepackage{float}
\usepackage{chngpage}
\usepackage{calc}
\usepackage{bigints}
\usepackage{afterpage}
\usepackage{array}
\usepackage{booktabs}
\usepackage{rotating}
\usepackage{multirow}
\usepackage{adjustbox}
\usepackage{tabularx}
\usepackage{verbatim}
\usepackage{mathtools}
\usepackage{ragged2e}
\usepackage[makeroom]{cancel}
\usepackage{caption}
\usepackage{hyperref}
\usepackage{caption}
\usepackage{subcaption}
\usepackage{cleveref}
\usepackage{appendix}
\usepackage{pgfplots}
\usepackage{soul}
\pgfplotsset{compat=1.16}
\usepackage[english]{babel}
\usepackage{hyphenat}
\usepackage[makeindex]{imakeidx}
\usetikzlibrary{datavisualization}
\usetikzlibrary{matrix}
\usetikzlibrary{datavisualization.formats.functions}

\newtheorem{theorem}{Theorem}[section]

\newtheorem{corollary}[theorem]{Corollary}
\newtheorem{lemma}[theorem]{Lemma}

\newtheorem{remark}[theorem]{Remark}
\newtheorem{assumption}[theorem]{Assumption}

\makeatletter
\def\section{\@startsection {section}{1}{\z@}{3.25ex plus 1ex minus
		.2ex}{1.5ex plus .2ex}{\large\bf}}
\def\subsection{\@startsection{subsection}{2}{\z@}{3.25ex plus 1ex minus
		.2ex}{1.5ex plus .2ex}{\normalsize\bf}}
\@addtoreset{equation}{section} 
\makeatother

\title{Solution formula for the general birth-death chemical diffusion master equation}

\author{Alberto Lanconelli\thanks{Dipartimento di Scienze Statistiche Paolo Fortunati, Università di Bologna, Bologna, Italy. \textbf{e-mail}: alberto.lanconelli2@unibo.it} \and Berk Tan Perçin\thanks{Dipartimento di Scienze Statistiche Paolo Fortunati, Università di Bologna, Bologna, Italy. \textbf{e-mail}: berktan.percin2@unibo.it} \and Mauricio J. del Razo\thanks{Department of Mathematics and Computer Science, Freie Universität Berlin, Arnimallee 6, 14195 Berlin, Germany and Dutch Institute for Emergent Phenomena, 1090GL Amsterdam, The Netherlands \textbf{e-mail}: m.delrazo@fu-berlin.de}}

\date{}

\begin{document}
	
\maketitle

\bigskip

\begin{abstract}
	We propose a solution formula for chemical diffusion master equations of birth and death type. These equations, proposed and formalized in the recent paper \cite{delRazo}, aim at incorporating the spatial diffusion of molecules into the description provided by the classical chemical master equation. We start from the general approach developed in \cite{Lanconelli} and perform a more detailed analysis of the representation found there. This leads to a solution formula for birth-death chemical diffusion master equations which is expressed in terms of the solution to the reaction-diffusion partial differential equation associated with the system under investigation. Such representation also reveals a striking analogy with the solution to the classical birth-death chemical master equations. The solutions of our findings are also illustrated for several examples. 
\end{abstract}



\allowdisplaybreaks

\section{Introduction and statement of the main results}

The dynamics of biochemical processes in living cells are commonly understood as an interplay between the spatial transport (diffusion) of molecules and their chemical kinetics (reaction), both of which are inherently stochastic at the molecular scale. In the case of systems with small molecule numbers in spatially well-mixed settings, the diffusion is averaged out and the probabilistic dynamics are governed by the well-known chemical master equation (CME) \cite{gillespie1977exact,qian2010chemical, qian2021stochastic}. The CME can be seldom solved analytically \cite{jahnke2007solving}. However, solving a few simple cases analytically can bring valuable insight to the solutions of more complex cases. Alternatively, one can solve it by integrating stochastic trajectories with the Gillespie or tau-leap algorithms \cite{anderson2015stochastic, gillespie1977exact}, by approximation methods \cite{deuflhard2008adaptive,engblom2009spectral,munsky2006finite,schnoerr2017approximation} or even by deep learning approaches \cite{gupta2021deepcme, jiang2021neural}.

In the case of spatially inhomogeneous systems, where diffusion is not averaged out, one would expect to obtain a similar master equation. However, obtaining such an equation is plagued with mathematical difficulties, and although it was hinted in previous work \cite{doi1976second} and formulated for some specific systems \cite{schweitzer2003brownian}, it was not until recently that this was formalized into the so-called chemical diffusion master equation (CDME) \cite{delRazo,delRazo2}. The CDME changes a few paradigms that have not yet been explored thoroughly in stochastic chemical kinetics models. It combines continuous and discrete degrees of freedom, and it models reaction and diffusion as a joint stochastic process. It consists of an infinite sorted family of Fokker-Planck equations, where each level of the sorted family corresponds to a certain number of particles/molecules. The equations at each level describe the spatial diffusion of the corresponding set of particles, and they are coupled to each other via reaction operators, which change the number of particles in the system. The CDME is the theoretical backbone of reaction-diffusion processes, and thus, it is fundamental to model and understand biochemical processes in living cells, as well as to develop multiscale numerical methods \cite{del2018grand,flegg2012two,kostre2021coupling,smith2018spatially} and hybrid algorithms \cite{chen2014brownian,dibak2018msm,del2021multiscale}. The stochastic trajectories of the CDME can be often integrated using particle--based reaction--diffusion simulations \cite{ andrews2017smoldyn, hoffmann2019readdy}. However, analytic and approximate solutions have not yet been explored in detail. In this work, we work out a method to obtain an analytic solution of the CDME for a simple birth-death reaction system, with the aim to bring insight of the CDME solution of more complex systems. 

We consider a system of indistinguishable molecules of a chemical species $S$ which undergo 
\begin{itemize}
	\item \emph{diffusion} in the bounded open region $\mathbb{X}$ of $\mathbb{R}^3$;
	\item \emph{degradation} and \emph{creation} chemical reactions
	\begin{equation*}
	\mbox{(I)}\quad S\xrightarrow{\lambda_d(x)}\varnothing\quad\quad\quad\mbox{(II)}\quad \varnothing\xrightarrow{\lambda_c(x)}S,
	\end{equation*}
	where $\lambda_d(x)$ denotes the propensity for reaction (I) to occur for a particle located at position $x\in\mathbb{X}$ (i.e., the probability per unit of time for this particle to disappear) while $\lambda_c(x)$ is the propensity for a new particle to be created at position $x\in\mathbb{X}$ by reaction (II).
\end{itemize} 
To describe the evolution in time of such system the authors in \cite{delRazo,delRazo2} proposed a set of equations for the number and position of the molecules. Namely, for $t\geq 0$, $n\geq 1$ and $A\in\mathcal{B}(\mathbb{X}^n)$ they set
\begin{align*}
\mathcal{N}(t)&:=\mbox{ number of molecules at time $t$},\\
\rho_0(t)&:=\mathbb{P}(\mathcal{N}(t)=0)\\
\int_{A}\rho_n(t,x_1,...,x_n)dx_1\cdot\cdot\cdot dx_n&:=\mathbb{P}\left(\{\mathcal{N}(t)=n\}\cap\{(X_1(t),...,X_n(t))\in A\}\right);
\end{align*}
here, $dx_i$ stands for the three dimensional integration volume $dx^{(1)}_idx^{(2)}_idx^{(3)}_i$. Then, according to \cite{delRazo,delRazo2} the time evolution of the reaction-diffusion process described above is governed by the following infinite system of equations: 
\begin{align}\label{equation}
\begin{cases}
\begin{split}
\partial_t\rho_n(t,x_1,...,x_n)=&\sum_{i=1}^n\Delta_i\rho_n(t,x_1,...,x_n)\\
&+(n+1)\int_{\mathbb{X}}\lambda_d(y)\rho_{n+1}(t,x_1,...,x_n,y)dy\\
&-\sum_{i=1}^n\lambda_d(x_i)\rho_n(t,x_1,...,x_n)\\
&+\frac{1}{n}\sum_{i=1}^n\lambda_c(x_i)\rho_{n-1}(t,x_1,...,x_{i-1},x_{i+1},...,x_n)\\
&-\int_{\mathbb{X}}\lambda_c(y)dy\cdot\rho_n(t,x_1,...,x_n),\quad\quad\quad n\geq 0, t>0, (x_1,...,x_n)\in \mathbb{X}^n;
\end{split}
\end{cases}
\end{align}
where we agree on assigning value zero to the three sums above when $n=0$. The term
\begin{align*}
\sum_{i=1}^n\Delta_i\rho_n(t,x_1,...,x_n)
\end{align*}
in \eqref{equation} refers to spatial diffusion of the particles: here, 
\begin{align*}
    \Delta_i:=\partial^2_{x_i^{(1)}}+\partial^2_{x_i^{(2)}}+\partial^2_{x_i^{(3)}}
\end{align*}
stands for the three dimensional Laplace operator. We remark that to ease the notation we choose a driftless isotropic diffusion but the extension to the divergence-form second order partial differential operator
\begin{align*}
\mathtt{L}_{x_i}v:=\sum_{l,m=1}^3\partial_{x^{(l)}_i}\left(a_{lm}(x_i)\partial_{x^{(m)}_i}v\right)-\sum_{l=1}^3\partial_{x^{(l)}_i}\left(b_{l}(x_i)v\right),
\end{align*}
which models a general anisotropic diffusion with drift on $\mathbb{R}^3$, is readily obtained. The terms
\begin{align*}
(n+1)\int_{\mathbb{X}}\lambda_d(y)\rho_{n+1}(t,x_1,...,x_n,y)dy-\sum_{i=1}^n\lambda_d(x_i)\rho_n(t,x_1,...,x_n) 
\end{align*}
formalize gain and loss, respectively, due to reaction (I), while
\begin{align*}
\frac{1}{n}\sum_{i=1}^n\lambda_c(x_i)\rho_{n-1}(t,x_1,...,x_{i-1},x_{i+1},...,x_n)-\int_{\mathbb{X}}\lambda_c(y)dy\cdot\rho_n(t,x_1,...,x_n)
\end{align*}
relate to reaction (II). System \eqref{equation} is combined with initial and Neumann boundary conditions 
\begin{align}\label{initial}
\begin{cases}
\begin{split}
\rho_0(0)&=1;\\
\rho_n(0,x_1,...,x_n)&=0,\quad n\geq 1, (x_1,...,x_n)\in\mathbb{X}^n;\\
\partial_{\nu}\rho_n(t,x_1,...,x_n)&=0,\quad n\geq 1, t\geq 0, (x_1,...,x_n)\in\partial \mathbb{X}^n.
\end{split}
\end{cases}
\end{align}
The initial condition above states that there are no molecules in the system at time zero while the Neumann condition prevents flux through the boundary of $\mathbb{X}$, thus forcing the diffusion of the molecules inside $\mathbb{X}$. The symbol $\partial_{\nu}$ in \eqref{initial} stands for the directional derivative along the outer normal vector at the boundary of $\mathbb{X}^n$.

Aim of this note is to present the following solution formula for \eqref{equation}-\eqref{initial}. 

\begin{theorem}\label{final}
	Let $v$ be a classical solution of the problem
	\begin{align}\label{acf}
	\begin{cases}
	\partial_tv(t,x)=\Delta v(t,x)-\lambda_d(x)v(t,x)+\lambda_c(x), & t> 0, x\in \mathbb{X};\\
	v(0,x)=0,& x\in \bar{\mathbb{X}};\\
	\partial_{\nu}v(t,x)=0,& t\geq 0, x\in\partial\mathbb{X}.
	\end{cases}
	\end{align}
	Then, the  chemical diffusion master equation \eqref{equation} with initial and boundary conditions \eqref{initial} has a classical solution given by 
	\begin{align}\label{1*}
	\rho_0(t)=\mathbb{P}(\mathcal{N}(t)=0)=\exp\left\{-\int_{\mathbb{X}}v(t,x)dx\right\},\quad t\geq 0,	
	\end{align}
	and for $n\geq 1$
	\begin{align}\label{2*}
	\rho_n(t,x_1,...,x_n)=\exp\left\{-\int_{\mathbb{X}}v(t,x)dx\right\}\frac{1}{n!}v(t,x_1)\cdot\cdot\cdot v(t,x_n), \quad t\geq 0,(x_1,...,x_n)\in \mathbb{X}^n.
	\end{align}
\end{theorem}

To prove the validity of the representations \eqref{1*}-\eqref{2*} one can trivially differentiate the right hand sides with respect to $t$ and verify using \eqref{acf} that they indeed solve \eqref{equation}-\eqref{initial}. We will however provide in the next section a constructive derivation of the expressions \eqref{1*}-\eqref{2*} which is based on the general approach proposed in \cite{Lanconelli}; here, an infinite dimensional version of the moment generating function method, which is commonly utilized to solve analytically some chemical master equations (see for details \cite{McQuarrie}), is developed. These techniques are also employed in an ongoing work which consider chemical diffusion master equations with higher order reactions. 

\begin{remark}
It is important to highlight the striking similarities between the representation formulas \eqref{1*}-\eqref{2*} for the solution of the CDME \eqref{equation}-\eqref{initial} and the solution 
\begin{align}\label{solution CME}
	\varphi_n(t)=\frac{\left(\frac{\mathtt{c}}{\mathtt{d}}(1-e^{-\mathtt{d}t})\right)^n}{n!}e^{-\frac{\mathtt{c}}{\mathtt{d}}(1-e^{-\mathtt{d}t})},\quad t\geq 0,n\geq 0,
	\end{align}
of the corresponding (diffusion-free) birth-death chemical master equation
\begin{align}\label{CME}
\dot{\varphi}_n(t)=\mathtt{d}(n+1)\varphi_{n+1}(t)+\mathtt{c}\varphi_{n-1}(t)-\mathtt{d}n\varphi_{n}(t)-\mathtt{c}\varphi_{n}(t),
\end{align}
with initial condition
\begin{align}\label{CME initial}
\varphi_n(0)=\delta_{0n},\quad\mbox{for all }n\geq 0.
\end{align}
Equation \eqref{CME}-\eqref{CME initial} describes the evolution in time of the probability
\begin{align*}
\varphi_n(t):=\mathbb{P}(\emph{number of molecules at time $t$}=n)
\end{align*}
for the reactions
\begin{align*}
	\emph{(I)}\quad S\xrightarrow{\mathtt{d}}\varnothing\quad\quad\quad\emph{(II)}\quad \varnothing\xrightarrow{\mathtt{c}}S,
	\end{align*}
	with no molecules at time zero. Here, $\mathtt{d}$ and $\mathtt{c}$ are the \emph{stochastic rate constants} for degradation and creation reactions, respectively. (To see how \eqref{solution CME} is derived from \eqref{CME}-\eqref{CME initial} one can for instance use the moment generating function method: see \cite{McQuarrie} for details).
 We note that the function
\begin{align*}
	t\mapsto \frac{\mathtt{c}}{\mathtt{d}}(1-e^{-\mathtt{d}t}),
	\end{align*}
appearing in \eqref{solution CME} solves the deterministic rate equation
\begin{align}\label{rate equation}
	\begin{cases}
	\frac{d}{dt}v(t)=-\mathtt{d} v(t)+\mathtt{c},& t>0\\
	v(0)=0&.
	\end{cases}
	\end{align}
This establishes a perfect agreement between \eqref{acf},\eqref{1*},\eqref{2*}, i.e. representation of the solution for \eqref{equation}-\eqref{initial} and reaction-diffusion PDE, on one side and \eqref{solution CME},\eqref{rate equation}, i.e. representation of the solution for \eqref{CME}-\eqref{CME initial} and rate equation, on the other side. 
\end{remark}

\begin{corollary}\label{corollary}
	In the reaction-diffusion model described by the CDME \eqref{equation}-\eqref{initial}, conditioned on the event $\{\mathcal{N}(t)=n\}$ the positions of the molecules at time $t$ are independent and identically distributed with probability density function
	\begin{align*}
	p(t,x):=\frac{v(t,x)}{\int_{\mathbb{X}}v(t,x)dx},\quad x\in \mathbb{X}.
	\end{align*}
	Moreover,
	\begin{align*}
	    \mathbb{P}(\mathcal{N}(t)=n)=\frac{\left(\int_{\mathbb{X}}v(t,x)dx\right)^n}{n!}\exp\left\{-\int_{\mathbb{X}}v(t,x)dx\right\}.
	\end{align*}
\end{corollary}

\begin{proof}
	Let $A\in\mathcal{B}(\mathbb{X}^n)$; then,
	\begin{align*}
	\mathbb{P}((X_1(t),...,X_n(t))\in A|\mathcal{N}(t)=n)&=\frac{\mathbb{P}(\{(X_1(t),...,X_n(t))\in A\}\cap\{\mathcal{N}(t)=n\})}{\mathbb{P}(\mathcal{N}(t)=n)}\\
	&=\frac{\int_{A}\rho_n(t,x_1,...,x_n)dx_1\cdot\cdot\cdot dx_n}{\int_{\mathbb{X}^n}\rho_n(t,x_1,...,x_n)dx_1\cdot\cdot\cdot dx_n}\\
	&=\frac{\int_{A}\exp\left\{-\int_{\mathbb{X}}v(t,x)dx\right\}\frac{1}{n!}v(t,x_1)\cdot\cdot\cdot v(t,x_n)dx_1\cdot\cdot\cdot dx_n}{\int_{\mathbb{X}^n}\exp\left\{-\int_{\mathbb{X}}v(t,x)dx\right\}\frac{1}{n!}v(t,x_1)\cdot\cdot\cdot v(t,x_n)dx_1\cdot\cdot\cdot dx_n}\\
	&=\int_A\frac{v(t,x_1)}{\int_{\mathbb{X}}v(t,x)dx}\cdot\cdot\cdot\frac{v(t,x_n)}{\int_{\mathbb{X}}v(t,x)dx}dx_1\cdot\cdot\cdot dx_n.
	\end{align*}
	The second part of the statement is proved as follows:
	\begin{align*}
	   \mathbb{P}(\mathcal{N}(t)=n)&=\int_{\mathbb{X}^n}\rho_n(t,x_1,...,x_n)dx_1\cdot\cdot\cdot dx_n\\
	   &=\int_{\mathbb{X}^n}\exp\left\{-\int_{\mathbb{X}}v(t,x)dx\right\}\frac{1}{n!}v(t,x_1)\cdot\cdot\cdot v(t,x_n)dx_1\cdot\cdot\cdot dx_n\\
	   &=\frac{\left(\int_{\mathbb{X}}v(t,x)dx\right)^n}{n!}\exp\left\{-\int_{\mathbb{X}}v(t,x)dx\right\}.
	\end{align*}
\end{proof}

The paper is organized as follows: in Section 2 we propose a constructive proof of Theorem \ref{final} which is based on the approach described in \cite{Lanconelli} while in Section 3 we show graphical illustrations of our findings for some particular cases of physical interest that allow for explicit computations in the reaction diffusion PDE \eqref{acf}. 

\section{Constructive proof of Theorem \ref{final}}\label{proof}

In this section we propose a constructive method to derive the representation formulas \eqref{1*}-\eqref{2*} of Theorem \ref{final}. The method we propose steams from a further development of the ideas and results presented in \cite{Lanconelli} which are reported here for easiness of reference. \\
For notational purposes we assume $\mathbb{X}=]0,1[$. Consider the birth-death CDME
\begin{align}\label{equation proof}
\begin{cases}
\begin{split}
\partial_t\rho_n(t,x_1,...,x_n)=&\sum_{i=1}^n\partial^2_{x_i}\rho_n(t,x_1,...,x_n)\\
&+(n+1)\int_0^1\lambda_d(y)\rho_{n+1}(t,x_1,...,x_n,y)dy\\
&-\sum_{i=1}^n\lambda_d(x_i)\rho_n(t,x_1,...,x_n)\\
&+\frac{1}{n}\sum_{i=1}^n\lambda_c(x_i)\rho_{n-1}(t,x_1,...,x_{i-1},x_{i+1},...,x_n)\\
&-\int_0^1\lambda_c(y)dy\cdot\rho_n(t,x_1,...,x_n),\quad\quad\quad n\geq 0, t>0, (x_1,...,x_n)\in ]0,1[^n,
\end{split}
\end{cases}
\end{align}
with the usual agreement of assigning value zero to the three sums above when $n=0$, together with initial and Neumann boundary conditions 
\begin{align}\label{initial proof}
\begin{cases}
\begin{split}
\rho_0(0)&=1;\\
\rho_n(0,x_1,...,x_n)&=0,\quad n\geq 1, (x_1,...,x_n)\in[0,1]^n;\\
\partial_{\nu}\rho_n(t,x_1,...,x_n)&=0,\quad n\geq 1, t\geq 0, (x_1,...,x_n)\in\partial [0,1]^n.
\end{split}
\end{cases}
\end{align}
We set
\begin{align}\label{A}
\mathcal{A}:=-\partial^2_x+\lambda_d(x),\quad x\in [0,1],
\end{align}
with homogenous Neumann boundary conditions and write $\{\xi_k\}_{k\geq 1}$ for the  orthonormal basis of $L^2([0,1])$ that diagonalizes the operator $\mathcal{A}$; this means that for all $j,k\geq 1$ we have
\begin{align*}
\int_0^1\xi_k(y)\xi_j(y)dy=\delta_{kj},\quad\xi_k'(0)=\xi_k'(1)=0,
\end{align*}
and there exists a sequence of non negative real numbers $\{\alpha_k\}_{k\geq 1}$ such that
\begin{align*}
\mathcal{A}\xi_k=\alpha_k\xi_k,\quad\mbox{ for all $k\geq 1$}.
\end{align*}
We observe that $\mathcal{A}$ is an unbounded, non negative self-adjoint operator. 
\begin{assumption}\label{assumption 1}
The sequence of eigenvalues $\{\alpha_k\}_{k\geq 1}$ is strictly positive.
\end{assumption} 
We now denote by $\Pi_N:L^2([0,1])\to L^2([0,1])$ the orthogonal projection onto the finite dimensional space spanned by $\{\xi_1,...,\xi_N\}$, i.e. 
\begin{align*}
\Pi_Nf(x):=\sum_{k=1}^N\langle f,\xi_k\rangle_{L^2([0,1])}\xi_k(x),\quad x\in [0,1];
\end{align*}
we also set
\begin{align}\label{notation lambda}
d_k:=\langle \lambda_d,\xi_k\rangle_{L^2([0,1])},\quad c_k:=\langle \lambda_c,\xi_k\rangle_{L^2([0,1])},\quad \gamma:=\int_0^1\lambda_c(y)dy.
\end{align} 

\begin{assumption}\label{assumption 2}
There exists $N_0\geq 1$ such that $\Pi_{N_0}\lambda_d=\lambda_d$; this is equivalent to say $\Pi_{N}\lambda_d=\lambda_d$ for all $N\geq N_0$.
\end{assumption}

In the sequel we set $\Pi_N^{\otimes n}$ to be the orthogonal projection from $L^2([0,1]^n)$ to the linear space generated by the functions $\{\xi_{i_1}\otimes\cdot\cdot\cdot\otimes\xi_{i_n}, 1\leq i_1,...,i_n\leq N\}$. The next theorem was proved in \cite{Lanconelli}.

\begin{theorem}\label{main theorem}
Let Assumptions \ref{assumption 1}-\ref{assumption 2} be in force and denote by $\{\rho_n\}_{n\geq 0}$ a classical solution of equation \eqref{equation}-\eqref{initial}. Then, for any $N\geq N_0$ and $t\geq 0$ we have the representation
\begin{align}\label{formula n=0}
	\rho_0^{(N)}(t)=\mathbb{E}[u_N(t,Z)],
\end{align}	
and for any $n\geq 1$ and $(x_1,...,x_n)\in [0,1]^n$,
\begin{align}\label{formula}
\Pi_N^{\otimes n}\rho_n(t,x_1,...,x_n)=\frac{1}{n!}\sum_{j_1,...j_n=1}^N\mathbb{E}\left[\left(\partial_{z_{j_1}}\cdot\cdot\cdot \partial_{z_{j_n}}u_N\right)(t,Z)\right]\xi_{j_1}(x_1)\cdot\cdot\cdot\xi_{j_n}(x_n).
\end{align}
Here, 
\begin{align}\label{formula-expectation}
\mathbb{E}\left[\left(\partial_{z_{j_1}}\cdot\cdot\cdot \partial_{z_{j_n}}u_N\right)(t,Z)\right]=\int_{\mathbb{R}^N}\left(\partial_{z_{j_1}}\cdot\cdot\cdot \partial_{z_{j_n}}u_N\right)(t,z)(2\pi)^{-N/2}e^{-\frac{|z|^2}{2}}dz,
\end{align}	
while $u_N:[0,+\infty[\times\mathbb{R}^N\to\mathbb{R}$ is a classical solution of the partial differential equation
\begin{align}\label{PDE}
\begin{cases}
\begin{split}
&\partial_tu_N(t,z)=\sum_{k=1}^N\alpha_k\partial^2_{z_k}u_N(t,z)+\sum_{k=1}^N\left(d_k-c_k-\alpha_k z_k\right)\partial_{z_k}u_N(t,z)+\left(\sum_{k=1}^Nc_kz_k-\gamma\right)u_N(t,z)\\
&u_N(0,z)=1,\quad t\geq 0, z\in\mathbb{R}^N.
\end{split}
\end{cases}
\end{align}
\end{theorem}

We now start working out the details of formulas \eqref{formula n=0}-\eqref{formula}.

\begin{lemma}
The solution to the Cauchy problem \eqref{PDE} can be represented as	
\begin{align}\label{b}
u_N(t,z)=\exp\left\{-\gamma t+\sum_{k=1}^N\left(c_kz_kg_k(t)+c_k(d_k-c_k)\int_0^tg_k(s)ds+c_k^2\alpha_k\int_0^tg_k(s)^2ds\right)\right\},
\end{align}
where
\begin{align}\label{c}
g_k(t):=\frac{1-e^{-\alpha_kt}}{\alpha_k},\quad t\geq 0,k=1,...,N.
\end{align}
\end{lemma}

\begin{proof}
The solution to the Cauchy problem \eqref{PDE} admits the following Feynman-Kac representation (see for instance \cite{KS})
\begin{align}\label{FK}
u_N(t,z)=\mathbb{E}\left[\exp\left\{\int_0^t\left(\sum_{k=1}^Nc_k\mathcal{Z}^{z_k}_k(s)-\gamma\right)ds\right\}\right],\quad t\geq 0, z=(z_1,...,z_N)\in\mathbb{R}^N.
\end{align} 
Here, for $k\in\{1,...,N\}$, the stochastic process $\{\mathcal{Z}_k^{z_k}(t)\}_{t\geq 0}$ is the unique strong solution of the mean-reverting Ornstein-Uhlenbeck stochastic differential equation
\begin{align}\label{OU}
d\mathcal{Z}_k^{z_k}(t)=\left(d_k-c_k-\alpha_k \mathcal{Z}^{z_k}_k(t)\right)dt+\sqrt{2\alpha_k}dW_k(t),\quad \mathcal{Z}_k^{z_k}(0)=z_k, 	
\end{align}
with $\{W_1(t)\}_{t\geq 0}$,...,$\{W_N(t)\}_{t\geq 0}$ being independent one dimensional Brownian motions. Using the independence of the processes $\mathcal{Z}_1^{z_1}$,...., $\mathcal{Z}_N^{z_N}$ we can rewrite \eqref{FK} as
\begin{align}\label{a}
u_N(t,z)&=e^{-\gamma t}\mathbb{E}\left[\exp\left\{\sum_{k=1}^Nc_k\int_0^t\mathcal{Z}^{z_k}_k(s)ds\right\}\right]=e^{-\gamma t}\mathbb{E}\left[\prod_{k=1}^N\exp\left\{c_k\int_0^t\mathcal{Z}^{z_k}_k(s)ds\right\}\right]\nonumber\\
&=e^{-\gamma t}\prod_{k=1}^N\mathbb{E}\left[\exp\left\{c_k\int_0^t\mathcal{Z}^{z_k}_k(s)ds\right\}\right].
\end{align}
We now want to compute the last expectation explicitly: first of all, we observe that equation \eqref{OU} admits the unique strong solution 
\begin{align*}
\mathcal{Z}^{z_k}_k(t)&=z_ke^{-\alpha_k t}+\frac{d_k-c_k}{\alpha_k}\left(1-e^{-\alpha_k t}\right)+\int_0^te^{-\alpha_k(t-s)}\sqrt{2\alpha_k}dW_k(s),
\end{align*}
(recall Assumption \ref{assumption 1}). Therefore,
\begin{align*}
\int_0^t\mathcal{Z}^{z_k}_k(s)ds&=z_k\frac{1-e^{-\alpha_kt}}{\alpha_k}+(d_k-c_k)\int_0^t\frac{1-e^{-\alpha_k s}}{\alpha_k}ds+\int_0^t\int_0^se^{-\alpha_k(s-u)}\sqrt{2\alpha_k}dW_k(u)ds\\
&=z_k\frac{1-e^{-\alpha_kt}}{\alpha_k}+(d_k-c_k)\int_0^t\frac{1-e^{-\alpha_k s}}{\alpha_k}ds+\sqrt{2\alpha_k}\int_0^t\frac{1-e^{-\alpha_k(t-s)}}{\alpha_k}dW_k(s);
\end{align*}
in the last equality we employed Fubini theorem for Lebesgue-Wiener integrals. The identity above yields
\begin{align*}
\mathbb{E}\left[\exp\left\{c_k\int_0^t\mathcal{Z}^{z_k}_k(s)ds\right\}\right]=&\exp\left\{c_k\left(z_k\frac{1-e^{-\alpha_kt}}{\alpha_k}+(d_k-c_k)\int_0^t\frac{1-e^{-\alpha_k s}}{\alpha_k}ds\right)\right\}\\
&\times\mathbb{E}\left[\exp\left\{c_k\sqrt{2\alpha_k}\int_0^t\frac{1-e^{-\alpha_k(t-s)}}{\alpha_k}dW_k(s)\right\}\right]\\
=&\exp\left\{c_k\left(z_k\frac{1-e^{-\alpha_kt}}{\alpha_k}+(d_k-c_k)\int_0^t\frac{1-e^{-\alpha_k s}}{\alpha_k}ds\right)\right\}\\
&\times\exp\left\{c_k^2\alpha_k\int_0^t\left(\frac{1-e^{-\alpha_k(t-s)}}{\alpha_k}\right)^2ds\right\},
\end{align*}
where in last equality we used the fact that $\int_0^t\frac{1-e^{-\alpha_k(t-s)}}{\alpha_k}dW_k(s)$ is a Gaussian random variable with mean zero and variance $\int_0^t\left(\frac{1-e^{-\alpha_k(t-s)}}{\alpha_k}\right)^2ds$.
This, together with \eqref{a}, gives
\begin{align*}
u_N(t,z)=&e^{-\gamma t}\prod_{k=1}^N\exp\left\{c_k\left(z_k\frac{1-e^{-\alpha_kt}}{\alpha_k}+(d_k-c_k)\int_0^t\frac{1-e^{-\alpha_k s}}{\alpha_k}ds\right)\right\}\nonumber\\
&\times\prod_{k=1}^N\exp\left\{c_k^2\alpha_k\int_0^t\left(\frac{1-e^{-\alpha_k(t-s)}}{\alpha_k}\right)^2ds\right\}\nonumber\\
=&\exp\left\{-\gamma t+\sum_{k=1}^N\left(c_kz_kg_k(t)+c_k(d_k-c_k)\int_0^tg_k(s)ds+c_k^2\alpha_k\int_0^tg_k(s)^2ds\right)\right\},
\end{align*}
(recall definition \eqref{c}). The proof is complete.
\end{proof}

\begin{lemma}
Expectation \eqref{formula n=0} can be written as
\begin{align*}
\rho_0^{(N)}(t)=&\exp\left\{t\left(\sum_{k=1}^N\frac{c_kd_k}{\alpha_k}-\gamma\right)+\sum_{k=1}^Nc_kd_k\frac{e^{-\alpha_k t}-1}{\alpha_k^2}\right\}.
\end{align*}
In particular,
\begin{align}\label{1}
\rho_0(t)=\lim_{N\to +\infty}\rho_0^{(N)}(t)=\exp\left\{\sum_{k\geq 1}c_kd_k\frac{e^{-\alpha_k t}-1}{\alpha_k^2}\right\}.
\end{align}
\end{lemma}

\begin{proof}
Let $Z=(Z_1,...,Z_N)$ be an $N$-dimensional vector of i.i.d. standard Gaussian random variables; then,
\begin{align*}
\rho_0^{(N)}(t)=&\mathbb{E}[u_N(t,Z)]\\
=&\mathbb{E}\left[\exp\left\{-\gamma t+\sum_{k=1}^N\left(c_kZ_ig_k(t)+c_k(d_k-c_k)\int_0^tg_k(s)ds+c_k^2\alpha_k\int_0^tg_k(s)^2ds\right)\right\}\right]\\
=&\exp\left\{-\gamma t+\sum_{k=1}^N\left(c_k(d_k-c_k)\int_0^tg_k(s)ds+c_k^2\alpha_k\int_0^tg_k(s)^2ds\right)\right\}\mathbb{E}\left[\exp\left\{\sum_{k=1}^Nc_kZ_ig_k(t)\right\}\right]\\
=&\exp\left\{-\gamma t+\sum_{k=1}^N\left(\frac{c_k^2g_k(t)^2}{2}+c_k(d_k-c_k)\int_0^tg_k(s)ds+c_k^2\alpha_k\int_0^tg_k(s)^2ds\right)\right\}\\
=&\exp\left\{-\gamma t+\sum_{k=1}^N\left[c_k^2\left(\frac{g_k(t)^2}{2}-\int_0^tg_k(s)ds+\alpha_k\int_0^tg_k(s)^2ds\right)+c_kd_k\int_0^tg_k(s)ds\right]\right\}\\
=&\exp\left\{-\gamma t+\sum_{k=1}^Nc_kd_k\int_0^tg_k(s)ds\right\}.
\end{align*}
The fourth equality follows from the expression of the exponential generating function of a Gaussian vector while the last equality is due to identity
\begin{align*}
\frac{g_k(t)^2}{2}-\int_0^tg_k(s)ds+\alpha_k\int_0^tg_k(s)^2ds=0,\quad t\geq 0
\end{align*}
which follows from a direct verification (recall definition \eqref{c}). On the other hand, we have
\begin{align*}
	\int_0^tg_k(s)ds=\frac{t}{\alpha_k}+\frac{e^{-\alpha_k t}-1}{\alpha_k^2},
\end{align*}
and hence
\begin{align*}
	\rho_0^{(N)}(t)=&\exp\left\{t\left(\sum_{k=1}^N\frac{c_kd_k}{\alpha_k}-\gamma\right)+\sum_{k=1}^Nc_kd_k\frac{e^{-\alpha_k t}-1}{\alpha_k^2}\right\}.
\end{align*}
Moreover, letting $N$ to infinity we get
\begin{align*}
\rho_0(t)=&\lim_{N\to +\infty}\rho_0^{(N)}(t) \\
=&\lim_{N\to +\infty}\exp\left\{t\left(\sum_{k=1}^N\frac{c_kd_k}{\alpha_k}-\gamma\right)+\sum_{k=1}^Nc_kd_k\frac{e^{-\alpha_k t}-1}{\alpha_k^2}\right\}\\
=&\exp\left\{t\left(\sum_{k\geq 1}\frac{c_kd_k}{\alpha_k}-\gamma\right)+\sum_{k\geq 1}c_kd_k\frac{e^{-\alpha_k t}-1}{\alpha_k^2}\right\}\\
=&\exp\left\{\sum_{k\geq 1}c_kd_k\frac{e^{-\alpha_k t}-1}{\alpha_k^2}\right\}.
\end{align*}
Here, we employed the identity
\begin{align*}
\sum_{k\geq 1}\frac{c_kd_k}{\alpha_k}=\gamma,
\end{align*}
which follows from
\begin{align*}
\sum_{k\geq 1}\frac{c_kd_k}{\alpha_k}=&\langle\mathcal{A}^{-1}\lambda_c,\lambda_d\rangle_{L^2([0,1])}=\langle\mathcal{A}^{-1}\lambda_c,\mathcal{A}\mathtt{1}\rangle_{L^2([0,1])}=\langle\lambda_c,\mathcal{A}^{-1}\mathcal{A}\mathtt{1}\rangle_{L^2([0,1])}\\
=&\langle\lambda_c,\mathtt{1}\rangle_{L^2([0,1])}=\int_0^1\lambda_c(x)\mathtt{1}(x)dx=\gamma.
\end{align*}
We also denoted $\mathtt{1}(x)=1$, $x\in [0,1]$ and exploited the identity $\mathcal{A}\mathtt{1}=\lambda_d$. 
\end{proof}

\begin{lemma}
Expectation \eqref{formula} can be written as
	\begin{align*}
	\Pi_N^{\otimes n}\rho_n(t,x_1,...,x_n)=\rho_0^{(N)}(t)\frac{1}{n!}\left(\sum_{j=1}^Nc_{j}g_{j}(t)\xi_{j}(x_1)\right)\cdot\cdot\cdot\left(\sum_{j=1}^Nc_{j}g_{j}(t)\xi_{j}(x_n)\right).
	\end{align*}
In particular,
\begin{align}\label{2}
\rho_n(t,x_1,...,x_n)&=\lim_{N\to+\infty}\Pi_N^{\otimes n}\rho_n(t,x_1,...,x_n)\nonumber\\
&=\exp\left\{-\sum_{k\geq 1}c_kd_k\frac{1-e^{-\alpha_k t}}{\alpha_k^2}\right\}\frac{1}{n!}\left(\sum_{j\geq 1}c_{j}\frac{1-e^{-\alpha_j t}}{\alpha_j}\xi_{j}\right)^{\otimes n}(x_1,...,x_n).
\end{align}
\end{lemma}

\begin{proof}
We note that according to \eqref{b} we have
\begin{align*}
\left(\partial_{z_{j_1}}\cdot\cdot\cdot \partial_{z_{j_n}}u_N\right)(t,z)=u_N(t,z) c_{j_1}g_{j_1}(t)\cdot\cdot\cdot c_{j_n}g_{j_n}(t),
\end{align*}
and hence
\begin{align*}
\mathbb{E}\left[\left(\partial_{z_{j_1}}\cdot\cdot\cdot \partial_{z_{j_n}}u_N\right)(t,Z)\right]=\rho_0^{(N)}(t)c_{j_1}g_{j_1}(t)\cdot\cdot\cdot c_{j_n}g_{j_n}(t).
\end{align*}
Therefore,
\begin{align*}
\Pi_N^{\otimes n}\rho_n(t,x_1,...,x_n)&=\frac{1}{n!}\sum_{j_1,...j_n=1}^N\mathbb{E}\left[\left(\partial_{z_{j_1}}\cdot\cdot\cdot \partial_{z_{j_n}}u_N\right)(t,Z)\right]\xi_{j_1}(x_1)\cdot\cdot\cdot\xi_{j_n}(x_n)\\
&=\frac{1}{n!}\sum_{j_1,...j_n=1}^N\rho_0^{(N)}(t) c_{j_1}g_{j_1}(t)\cdot\cdot\cdot c_{j_n}g_{j_n}(t)\xi_{j_1}(x_1)\cdot\cdot\cdot\xi_{j_n}(x_n)\\
&=\rho_0^{(N)}(t)\frac{1}{n!}\left(\sum_{j=1}^Nc_{j}g_{j}(t)\xi_{j}(x_1)\right)\cdot\cdot\cdot\left(\sum_{j=1}^Nc_{j}g_{j}(t)\xi_{j}(x_n)\right).
\end{align*}
Moreover, letting $N$ to infinity we obtain 
\begin{align*}
\rho_n(t,x_1,...,x_n)&=\exp\left\{-\sum_{k\geq 1}c_kd_k\frac{1-e^{-\alpha_k t}}{\alpha_k^2}\right\}\frac{1}{n!}\left(\sum_{j\geq 1}c_{j}g_{j}(t)\xi_{j}(x_1)\right)\cdot\cdot\cdot\left(\sum_{j\geq 1}c_{j}g_{j}(t)\xi_{j}(x_n)\right)\\
&=\exp\left\{-\sum_{k\geq 1}c_kd_k\frac{1-e^{-\alpha_k t}}{\alpha_k^2}\right\}\frac{1}{n!}\left(\sum_{j\geq 1}c_{j}\frac{1-e^{-\alpha_j t}}{\alpha_j}\xi_{j}\right)^{\otimes n}(x_1,...,x_n).
\end{align*}
\end{proof}

We are now in a position to show the equivalence between \eqref{1}-\eqref{2} and \eqref{1*}-\eqref{2*}. \\
	We start observing that
	\begin{align*}
	d_k:=\langle\lambda_d,\xi_k\rangle=\langle\mathcal{A}1,\xi_k\rangle=\langle 1,\mathcal{A}\xi_k\rangle=\alpha_k\langle 1,\xi_k\rangle=\alpha_k\int_0^1\xi_k(x)dx.
	\end{align*}
	Therefore, from formula \eqref{1} we can write 
	\begin{align*}
	\rho_0(t)=\mathbb{P}(\mathcal{N}(t)=0)&=\exp\left\{-\sum_{k\geq 1}c_kd_k\frac{1-e^{-\alpha_k t}}{\alpha_k^2}\right\}\\
	&=\exp\left\{-\sum_{k\geq 1}c_k\int_0^1\xi_k(x)dx\frac{1-e^{-\alpha_k t}}{\alpha_k}\right\}\\
	&=\exp\left\{-\int_0^1\left(\sum_{k\geq 1}c_k\frac{1-e^{-\alpha_k t}}{\alpha_k}\xi_k(x)\right)dx\right\}\\
	&=\exp\left\{-\int_0^1v(t,x)dx\right\},
	\end{align*}
	where we set
	\begin{align}\label{z}
	v(t,x):=\sum_{k\geq 1}c_k\frac{1-e^{-\alpha_k t}}{\alpha_k}\xi_k(x).
	\end{align}
	Note that with this notation we can also write according to \eqref{2} that 
	\begin{align*}
	\rho_n(t,x_1,...,x_n)=\exp\left\{-\int_{\mathbb{X}}v(t,x)dx\right\}\frac{1}{n!}v(t,x_1)\cdot\cdot\cdot v(t,x_n), \quad t\geq 0,(x_1,...,x_n)\in \mathbb{X}^n.
	\end{align*}
	If we now prove that the function $v$ defined in \eqref{z} solves \eqref{acf}, then the equivalence between \eqref{1}-\eqref{2} and \eqref{1*}-\eqref{2*} will be established. Since
	\begin{align*}
	\partial_tv(t,x)&=\partial_t\left(\sum_{k\geq 1}c_k\frac{1-e^{-\alpha_k t}}{\alpha_k}\xi_k(x)\right)=\sum_{k\geq 1}c_ke^{-\alpha_k t}\xi_k(x),
	\end{align*}
	we can conclude that 
	\begin{align*}
	\partial_x^2v(t,x)-\lambda_d(x)v(t,x)&=-\mathcal{A}v(t,x)=-\mathcal{A}\left(\sum_{k\geq 1}c_k\frac{1-e^{-\alpha_k t}}{\alpha_k}\xi_k(x)\right)=\sum_{k\geq 1}c_k(e^{-\alpha_k t}-1)\xi_k(x)\\
	&=\sum_{k\geq 1}c_ke^{-\alpha_k t}\xi_k(x)-\lambda_c(x)=\partial_tv(t,x)-\lambda_c(x),
	\end{align*}
	proving the desired property (the initial and boundary conditions in \eqref{acf} are readily satisfied).

\section{Case study: one dimensional motion with constant degradation function}

In this section we illustrate through several plots our theoretical findings for some concrete models. According to formulas \eqref{1*}-\eqref{2*} the solution to the chemical diffusion master equation \eqref{equation}-\eqref{initial} is completely determined by the solution of equation \eqref{acf}. To solve this problem explicitly we decided to focus on the one dimensional case $\mathbb{X}=]0,1[$ with driftless isotropic diffusion (i.e. the framework of Section \ref{proof}) and constant degradation function $\lambda_d$. This last restriction yields the advantage of knowing the explicit form of the eigenfunctions and eigenvalues of the operator $\mathcal{A}$ in \eqref{A} and hence the possibility of working with \eqref{1}-\eqref{2}, which we recall to be equivalent to \eqref{1*}-\eqref{2*}.\\
When $\lambda_d(x)=\lambda_d$, $x\in [0,1]$ for some positive constant $\lambda_d$, we get
\begin{align*}
	\mathcal{A}f(x)=-f''(x)+\lambda_d f(x),\quad\quad \xi_k(x)=\cos((k-1)\pi x),\quad k\geq 1,
\end{align*} 	
and 
\begin{align}\label{alpha}
	\alpha_k=(k-1)^2\pi^2+\lambda_d,\quad k\geq 1.
\end{align}
Therefore, the degradation function $\lambda_d(x)=\lambda_d$ is proportional to the first eigenfunction $\xi_1(x)=\mathtt{1}(x)$ and hence orthogonal to all the other eigenfunctions $\xi_k(x)$ for $k\geq 2$; this gives
\begin{align*}
d_1=\langle \lambda_d,\xi_1\rangle=\lambda_d\quad\mbox{ and }\quad d_k=\langle \lambda_d,\xi_k\rangle=0\mbox{ for all }k\geq 2;
\end{align*}
note also that \eqref{alpha} implies $\alpha_1=\lambda_d$. Combining these facts in \eqref{1} and \eqref{2} we obtain
\begin{align}\label{d1}
	\rho_0(t)=\mathbb{P}(\mathcal{N}(t)=0)=\exp\left\{-c_1\frac{1-e^{-\lambda_d t}}{\lambda_d}\right\},\quad t\geq 0,	
\end{align}
and, for $n\geq 1$, $t\geq 0$ and $(x_1,...,x_n)\in [0,1]^n$,
\begin{align}\label{d}
	\rho_n(t,x_1,...,x_n)=\exp\left\{-c_1\frac{1-e^{-\lambda_d t}}{\lambda_d}\right\}\frac{1}{n!}\left(\sum_{j\geq 1}c_{j}\frac{1-e^{-\alpha_j t}}{\alpha_j}\xi_{j}\right)^{\otimes n}(x_1,...,x_n).
\end{align}
We now specify some interesting choices of the creation function $\lambda_c$.

\subsection{Constant creation function}
In the case $\lambda_c(x)=\lambda_c$, $x\in [0,1]$ for some positive constant $\lambda_c$, in other words the creation is uniform in the whole interval just like the degradation, we get from \eqref{d1} and \eqref{d}
\begin{align*}
	\rho_0(t)=\mathbb{P}(\mathcal{N}(t)=0)=\exp\left\{-\lambda_c\frac{1-e^{-\lambda_d t}}{\lambda_d}\right\},\quad t\geq 0,	
\end{align*}

\begin{figure}[h!]
	\centering
	\begin{subfigure}[c]{0.42\textwidth}
		\includegraphics[width=\linewidth, height=5cm]{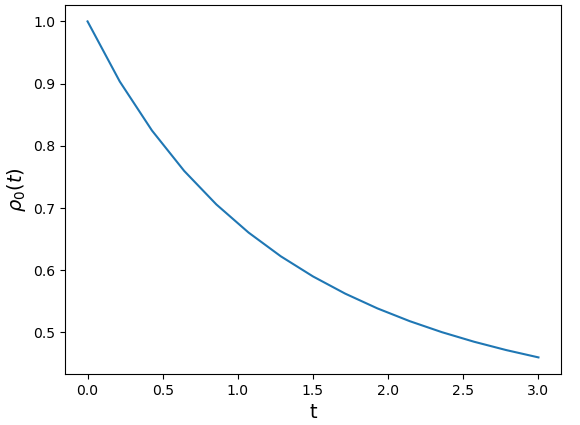}
		\caption{}
		\label{uniform creation 0 particles figure}
	\end{subfigure}
	\hspace{0.2cm}
	\begin{subfigure}[c]{0.47\textwidth}
		\includegraphics[width=\linewidth, height=5cm]{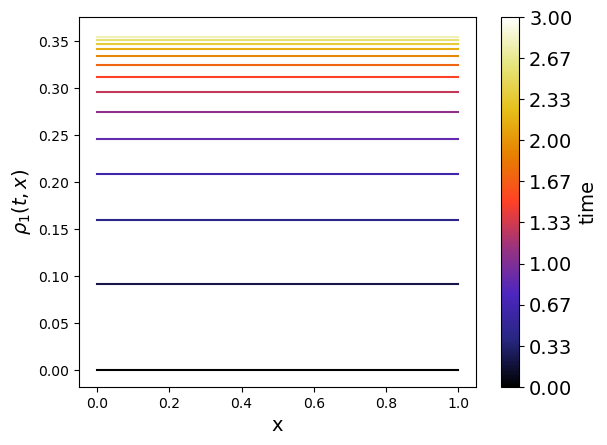}
		\caption{}
		\label{constant creation multiline figure}
	\end{subfigure}
	\caption{Solutions plots of the bdCDME generated with constant creation and degradation rates with $\lambda_c=\lambda_d=0.5$. a. The solution of $\rho_0(t)$ as a function of time. b. The solution of $\rho_1(t,x_1)$ as a function of position and time.}
 \label{fig:constant_creation}
\end{figure}

and for $n\geq 1$
\begin{align*}
	\rho_n(t,x_1,...,x_n)&=\exp\left\{-\lambda_c\frac{1-e^{-\lambda_d t}}{\lambda_d}\right\}\frac{1}{n!}\left(\lambda_c\frac{1-e^{-\lambda_d t}}{\lambda_d}\xi_1\right)^{\otimes n}(x_1,...,x_n)\\
	&=\exp\left\{-\lambda_c\frac{1-e^{-\lambda_d t}}{\lambda_d}\right\}\frac{1}{n!}\left(\lambda_c\frac{1-e^{-\lambda_d t}}{\lambda_d}\right)^{n}1(x_1)\cdot\cdot\cdot 1(x_n).
\end{align*}
This shows that for any $n\geq 1$ the function $\rho_n$ is constant in $x_1,...,x_n$ with height given by the $n$-th component of the  solution to the birth-death chemical master equation with stochastic rate constants $\lambda_d$ and $\lambda_c$ (compare with \eqref{solution CME}). 

Figure \ref{fig:constant_creation} shows the solutions for $\rho_0$ and $\rho_1$ as a function of time. Figure \ref{uniform creation 0 particles figure} shows the the exponential decay of the probability of having $0$ particles due to the constant creation of particles, and \ref{constant creation multiline figure} shows the probability distribution of having $1$ particle is uniform in space for all times, as well as its convergence to the stationary distribution. 

\subsection{Dirac delta creation function at $x=0$}
In this case we take $\lambda_c(x)=\lambda_c\delta_0(x)$, $x\in [0,1]$ for some positive constant $\lambda_c$, so the creation takes place only in the leftmost point of the interval while degradation happens uniformly. This way one yields
\begin{align*}
c_j=\int_0^1\lambda_c(x)\xi_j(x)dx=\int_0^1\lambda_c\delta_0(x)\xi_j(x)dx=\lambda_c\xi_j(0)=\lambda_c,\quad\mbox{ for all $j\geq 1$}.
\end{align*}
and formulas \eqref{d1} and \eqref{d} now read 
\begin{align*}
	\rho_0(t)=\mathbb{P}(\mathcal{N}(t)=0)=\exp\left\{-\lambda_c\frac{1-e^{-\lambda_d t}}{\lambda_d}\right\},\quad t\geq 0,	
\end{align*}

\begin{figure}[htb]
	\centering
	\begin{subfigure}[c]{0.45\textwidth}
		\includegraphics[width=\linewidth, height=5cm]{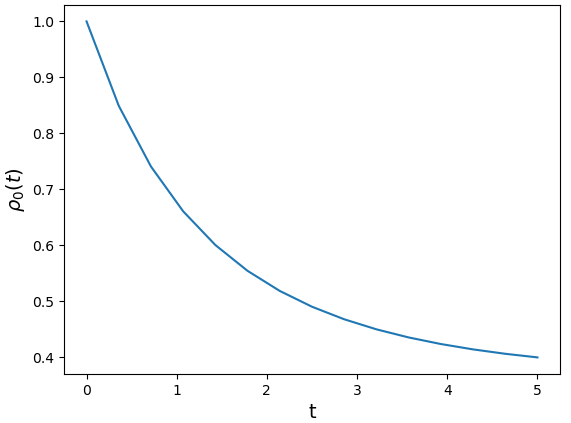}
		\caption{}
		\label{creation at 0 a)}
	\end{subfigure}
	\begin{subfigure}[c]{0.45\textwidth}
		\includegraphics[width=\linewidth, height=5cm]{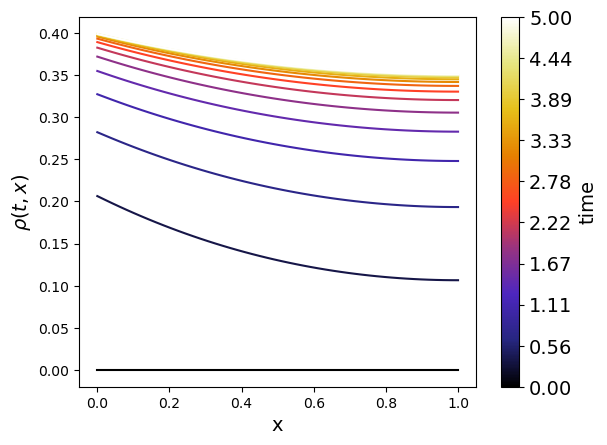}
		\caption{}
		\label{creation at 0 b)}
	\end{subfigure}
	\centering
	\begin{subfigure}[c]{0.55\textwidth}
		\includegraphics[width=\linewidth]{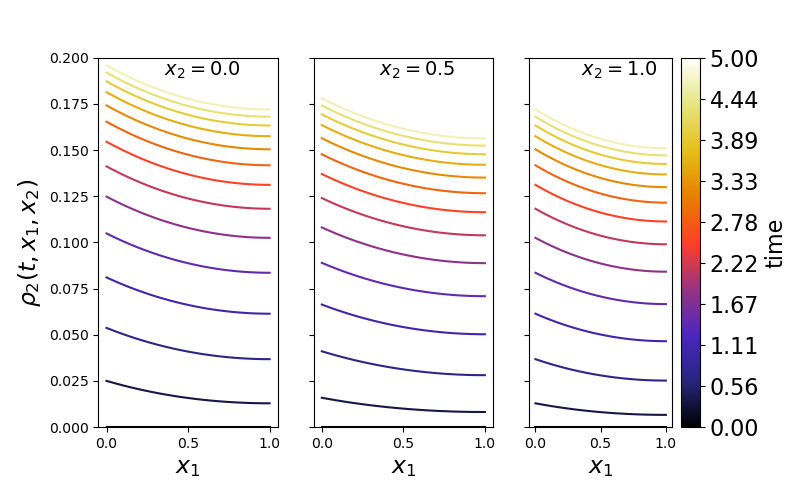}
		\caption{}
		\label{creation at 0 c)}
	\end{subfigure}
	\begin{subfigure}[c]{0.4\textwidth}
	    \includegraphics[width=\linewidth, height=5.5cm]{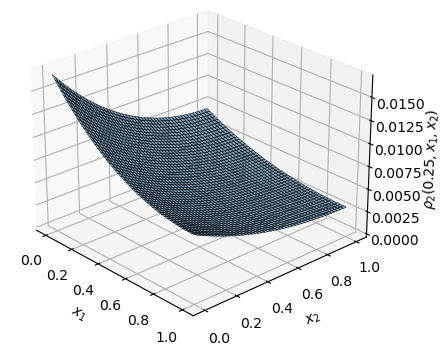}
	    \caption{}
	    \label{creation at 0 d)}
	\end{subfigure}
\caption{Solutions plots of the bdCDME generated with creation of particles at $x=0$ and constant degradation in the whole domain, namely with $\lambda_c(x)=0.5\delta_0(x)$ and $\lambda_d(x)=0.5$. The first 1000 terms of the sum in \cref{eq:genFunctSum} are considered. a. The solution of the $0$ particle density ($\rho_0(t)$) as a function of time. b. The solution of the  $1$ particle density ($\rho_1(t,x_1)$) for given position and time. c. The solution of the $2$ particle density ($\rho_2(t,x_1,x_2)$) with respect to $x_1$ and $t$ for three values of $x_2$. Time points as indicated in the color bar. d. The solution of the two particle density for fixed time, $\rho_2(t=0.25,x_1,x_2)$, as a function of $x_1$ and $x_2$.}
\label{creation at origin full figure}
\end{figure}
\clearpage
and for $n\geq 1$
\begin{align*}
	\rho_n(t,x_1,...,x_n)&=\exp\left\{-\lambda_c\frac{1-e^{-\lambda_d t}}{\lambda_d}\right\}\frac{1}{n!}\left(\sum_{j\geq 1}\lambda_c\frac{1-e^{-\alpha_j t}}{\alpha_j}\xi_{j}\right)^{\otimes n}(x_1,...,x_n)\\
	&=\exp\left\{-\lambda_c\frac{1-e^{-\lambda_d t}}{\lambda_d}\right\}\frac{\lambda_c^n}{n!}\left(\sum_{j\geq 1}\frac{1-e^{-\alpha_j t}}{\alpha_j}\xi_{j}\right)^{\otimes n}(x_1,...,x_n).
\end{align*}
We note that even though $\lambda_c(x)$ is a generalized function the series 
\begin{align}
\sum_{j\geq 1}c_{j}\frac{1-e^{-\alpha_j t}}{\alpha_j}\xi_{j}=\lambda_c\sum_{j\geq 1}\frac{1-e^{-\alpha_j t}}{\alpha_j}\xi_{j}
\label{eq:genFunctSum}
\end{align}
appearing above converges in $L^2([0,1])$.

In Figure \ref{creation at origin full figure} we plot solution of the bdCDME for this example. Figure \ref{creation at 0 a)} shows the exponential decay of the probability of having 0 particles due to the constant creation of particles. In contrast with Figure \ref{fig:constant_creation}, in Figures \ref{creation at 0 b)} and \ref{creation at 0 c)}, one can see the effect of the creation happening only at $x=0$ due to the peaks at the origin, while the highest peak is when $x_2=x_1=0$. With increasing time the peaks at the origin smooth out due to diffusion and probability being distributed through the different particle number densities. Similar to before, the curves converge to their stationary distribution as time increases. Lastly, Figure \ref{creation at 0 d)} shows the solution of the bdCDME for 2 particles as a surface on $x_1,x_2$ axes, when time is fixed at $t=0.25$. 

\subsection{Dirac delta creation function at $x=1/2$}

\begin{figure}[ht]
	\centering
	\begin{subfigure}[c]{0.45\textwidth}
		\includegraphics[width=\linewidth, height=5cm]{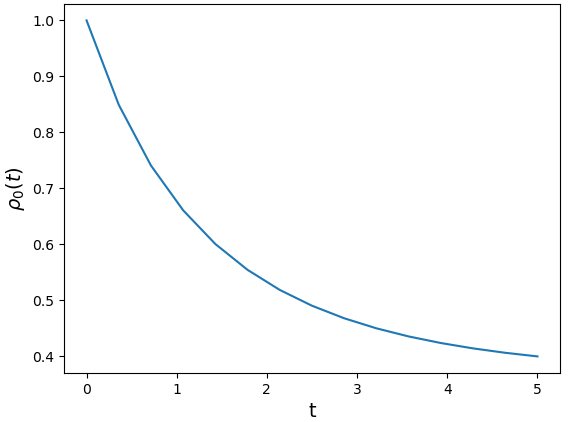}
		\caption{}
		\label{creation at 0.5 a)}
	\end{subfigure}
	\begin{subfigure}[c]{0.45\textwidth}
		\includegraphics[width=\linewidth, height=5cm]{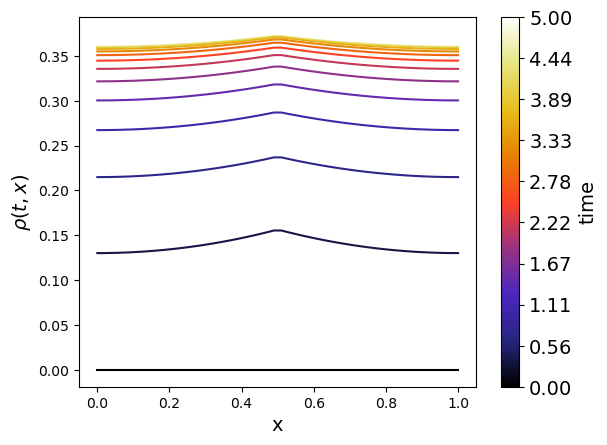}
		\caption{}
		\label{creation at 0.5 b)}
	\end{subfigure}
	\centering
	\begin{subfigure}[c]{0.55\textwidth}
		\includegraphics[width=\linewidth]{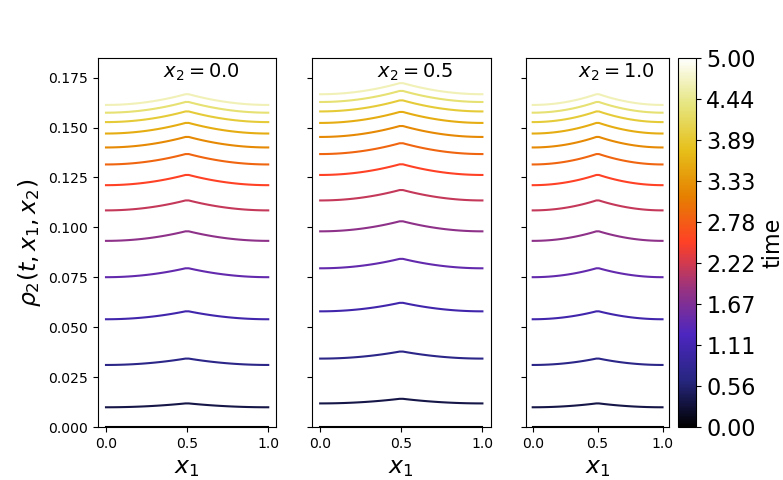}
		\caption{}
		\label{creation at 0.5 c)}
	\end{subfigure}
	\begin{subfigure}[c]{0.4\textwidth}
	    \includegraphics[width=\linewidth, height=5.5cm]{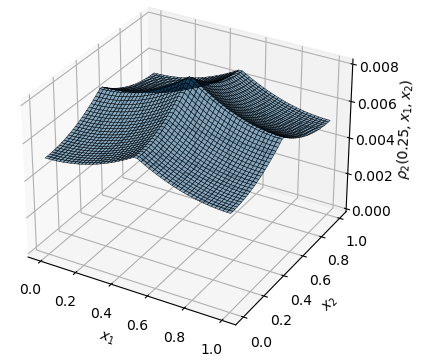}
	    \caption{}
	    \label{creation at 0.5 d)}
	\end{subfigure}
        \caption{Solutions plots of the bdCDME generated with creation of particles at $x=0.5$ and constant degradation in the whole domain, namely $\lambda_c=0.5\delta_1/2(x)$ and $\lambda_d=0.5$. The first 1000 terms of the sum in \cref{eq:genFunctSum} are considered. a. The solution of the $0$ particle density ($\rho_0(t)$) as a function of time. b. The solution of the $1$ particle density ($\rho_1(t,x_1)$) for given position and time. c. The solution of the $2$ particle density ($\rho_2(t,x_1,x_2)$) with respect to $x_1$ and $t$ for three values of $x_2$. Time points as indicated in the color bar. d. The solution of the two particle density for fixed time, $\rho_2(t=0.25,x_1,x_2)$, as a function of $x_1$ and $x_2$.}
	\label{creation at 0.5 full figure}
\end{figure}
We now choose $\lambda_c(x)=\lambda_c\delta_{1/2}(x)$, $x\in [0,1]$ for some positive constant $\lambda_c$, so the creation takes place only on the middle of the interval and degradation happens uniformly. This way one obtains

\begin{align*}
	c_j=\int_0^1\lambda_c(x)\xi_j(x)dx=\int_0^1\lambda_c\delta_{1/2}(x)\xi_j(x)dx=\lambda_c\xi_j(1/2)=\lambda_c\cos((j-1)\pi/2),\quad\mbox{ for all $j\geq 1$}.
\end{align*}
Therefore, equations \eqref{d1} and \eqref{d} take now the form
\begin{align*}
	\rho_0(t)=\mathbb{P}(\mathcal{N}(t)=0)=\exp\left\{-\lambda_c\frac{1-e^{-\lambda_d t}}{\lambda_d}\right\},\quad t\geq 0,	
\end{align*}

and for $n\geq 1$
\begin{align*}
	\rho_n(t,x_1,...,x_n)&=\exp\left\{-\lambda_c\frac{1-e^{-\lambda_d t}}{\lambda_d}\right\}\frac{1}{n!}\left(\sum_{j\geq 1}\lambda_c\cos((j-1)\pi/2)\frac{1-e^{-\alpha_j t}}{\alpha_j}\xi_{j}\right)^{\otimes n}(x_1,...,x_n)\\
	&=\exp\left\{-\lambda_c\frac{1-e^{-\lambda_d t}}{\lambda_d}\right\}\frac{\lambda_c^n}{n!}\left(\sum_{k\geq 1}(-1)^{k-1}\frac{1-e^{-\alpha_{2k-1} t}}{\alpha_{2k-1}}\xi_{2k-1}\right)^{\otimes n}(x_1,...,x_n).
\end{align*}

Figure \ref{creation at 0.5 full figure} shows plots of the solution of the bdCDME for this example. Figure \ref{creation at 0.5 a)} shows the exponential decay of the probability of having $0$ particles due to the constant creation of particles. However, in conrast with figures \ref{fig:constant_creation} and \ref{creation at origin full figure}, in this case, the effect of creation in the middle of the interval can be seen in the peaks in the Figures \ref{creation at 0.5 b)} and \ref{creation at 0.5 c)}, while the highest peak is at $x_1=x_2=0.5$, as expected. Similar to the previous example the effect of the location of the creation of particles on the distribution becomes less important with increasing time due to diffusion. Once again, the curves converge to their stationary distribution. Lastly, Figure \ref{creation at 0.5 d)} plots the solution of bdCDME as a surface for 2 particle case at fixed time $t=0.25$, as a function of $x_1$ and $x_2$.

\paragraph{Acknowledgments} M.J.R acknowledges the support from Deutsche Forschungsgemeinschaft (DFG) (Grant No. RA 3601/1-1) and from the Dutch Institute for Emergent Phenomena (DIEP) cluster at the University of Amsterdam.

\clearpage
\bibliographystyle{plain}
\bibliography{birth_death_CDME}

\end{document}